\documentclass[preprint,12pt]{elsarticle}

\usepackage{amssymb,amsmath,amsthm,graphicx}
\usepackage{enumerate}
\usepackage{float}
\usepackage{stmaryrd,mathbbol,mathrsfs} 
\usepackage{multicol}
\usepackage{float}
\usepackage{stackrel}
\usepackage{color}
\usepackage{tikz}
\usetikzlibrary{matrix,arrows}
\usepackage{capt-of}
\usepackage{array}

\newtheorem{theorem}{Theorem}[section]
\newtheorem*{theorem*}{Theorem}
\newtheorem{lemma}[theorem]{Lemma}

\newtheorem{corollary}[theorem]{Corollary}

\theoremstyle{remark}

\theoremstyle{definition}
\newtheorem{example}[theorem]{Example}
\newtheorem{definition}[theorem]{Definition}

\makeatletter
\newcommand{\Spvek}[2][r]{%
  \gdef\@VORNE{1}
  \left(\hskip-\arraycolsep%
    \begin{array}{#1}\vekSp@lten{#2}\end{array}%
  \hskip-\arraycolsep\right)}

\def\vekSp@lten#1{\xvekSp@lten#1;vekL@stLine;}
\def\vekL@stLine{vekL@stLine}
\def\xvekSp@lten#1;{\def\temp{#1}%
  \ifx\temp\vekL@stLine
  \else
    \ifnum\@VORNE=1\gdef\@VORNE{0}
    \else\@arraycr\fi%
    #1%
    \expandafter\xvekSp@lten
  \fi}
\makeatother
\journal{}

\begin{document}

\begin{frontmatter}



\title{On the clone of aggregation functions on bounded lattices\footnote{Preprint of an article published by Elsevier in the Information Sciences 329, Pages 381-389. It is available online at: \newline https://www.sciencedirect.com/science/article/pii/S0020025515006933}}


\author{Radom\'ir Hala\v{s}}
\address{Palack\'y University Olomouc, Faculty of Science, Department of Algebra and Geometry, 17. listopadu 12, 771 46 Olomouc, Czech Republic}
\ead{radomir.halas@upol.cz}

\author{Jozef P\'ocs}
\address{Palack\'y University Olomouc, Faculty of Science, Department of Algebra and Geometry, 17. listopadu 12, 771 46 Olomouc, Czech Republic\\ 
and\\
Mathematical Institute, Slovak Academy of Sciences,\\ Gre\v s\'akova 6, 040 01 Ko\v sice, Slovakia}
\ead{pocs@saske.sk}

\begin{abstract}
The main aim of this paper is to study aggregation functions on lattices via clone theory approach. Observing that the aggregation functions on lattices just correspond to $0,1$-monotone clones, as the main result we show that for any finite $n$-element lattice $L$ there is a set of at most $2n+2$ aggregation functions on $L$ from which the respective clone is generated. Namely, the set of generating aggregation functions consists only of at most $n$ unary functions, at most $n$ binary functions, and lattice operations 
$\wedge,\vee$, and all aggregation 
functions of $L$ are composed of them by usual term composition. Moreover, our approach works also for infinite lattices (such as mostly considered bounded real intervals $[a,b]$), where in contrast to finite case infinite suprema and (or, equivalently, a kind of limit process) have to be considered.
\end{abstract}


\end{frontmatter}

\section{Introduction}

Aggregation represents  a process of merging and combining several (usually numerical)
data in a single output. The mathematical theory of aggregation is based on the notion of an aggregation function describing the process of merging. 

Perhaps the most natural example of aggregation function is the arithmetic mean, which has been widely used in physics and all experimental sciences. In fact, aggregation functions appear in many branches of science where the fusion of information is requested, but more generally they appear in a pure mathematics (functional equations, theory of means and averages, measure and integration theory), in applied mathematics (probability, statistics, decision mathematics), computer and engineering sciences (artificial intelligence, operation research, data fusion, automatic reasoning etc.). During last decades they were successfully applied also in social sciences, economy, life sciences and other branches of research.

As the process of aggregation should represent in a certain sense the "synthesis" of input data, consequently the aggregation functions cannot be arbitrary and have to fulfill some natural minimal conditions. The statement that the output should be 
the synthetic value of inputs can be translated into the condition that the output value should lie in the same interval as the input ones, and, additionally, the least and the greatest values should be preserved.
The second natural widely accepted condition is nondecreasing monotonicity of the aggregation function, meaning that the increase of any of the input values should reflect this increase, or at worst, stay constant.

The theory of aggregation functions is quite well developed in case when the input (and, consequently, the output) values of 
these functions lie in a nonempty real interval $\mathbb I$, bounded or not. The formal definition is as follows:
\noindent
An aggregation function in $\mathbb I^n$ is a function $A: \mathbb I^n\mapsto \mathbb I$ that 
\begin{itemize}
\item [(i)] is nondecreasing (in each variable)
\item [(ii)] fulfills the boundary conditions
\begin{equation}\label{eq1}
\inf \, \{A(\mathbf{x}):\, \mathbf{x}\in \mathbb{I}^n\}=\inf\, \mathbb I \quad {\rm and} \quad \sup \, \{A(\mathbf{x}):\, \mathbf{x}\in \mathbb{I}^n\}=\sup \, \mathbb{I}.
\end{equation}
\end{itemize}

The integer $n$ represents the arity of the aggregation function.
For details one can refer the reader e.g. to the comprehensive 
monograph \cite{GMRP}.

Clearly, bounded real intervals can be viewed as special instances of (linearly ordered) lattices. Recall that a lattice is 
an algebra $(L;\vee,\wedge)$, where $L$ is a nonempty set with two binary operations $\vee$ and $\wedge$ representing 
suprema and infima. Let us mention that lattice theory is a very well established discipline of general algebra, there are 
several monographs on this topic, among them the most frequently used and quoted are the books by G. Gr\"atzer, \cite{G1,G2}.

Modifying the definition \eqref{eq1} of an aggregation function on a bounded real interval, quite recently and naturally the notion of aggregation function has been enlarged to bounded lattices. In contrast to real valued functions, not much is known on aggregation functions on lattices. We can mention e.g. the paper \cite{KM} devoted to a certain classification of aggregation functions on bounded partially ordered sets and lattices. 

One of the central problems connected with aggregation functions is how can they be constructed. We can mention several construction methods like transformed aggregation, composed aggregation, weighted aggregation, forming ordinal sums etc., for details we refer to the monograph \cite{GMRP}. Each of the above mentioned methods typically relies on a very specific approach and the methods usually have a quite different issue. For example, there is a group of methods characterized by the property "from simple to complex". In a classical case, the idea is based on standard arithmetical operations on the real line and fixed real functions. Another group of aggregation functions starts from given aggregation functions to construct the new ones.

The focus of our paper is a bit different. We start from a given lattice $L$ and ask the following question:
\begin{itemize}
\item[(Q)] Is there some uniform construction method of aggregation functions on $L$ (i.e., a method not depending on $L$) and a certain simple uniform set of aggregation functions on $L$ (i.e., functions of the same kind whose number depends only on the cardinality of the lattice $L$), such that applying the uniform method to this uniform set we obtain all the aggregation functions on $L$? 
\end{itemize}
By our knowledge, this approach to a study of aggregation functions on lattices is  new and promising better understanding of their construction methods. In this paper, by using specific methods of universal algebra, we give a positive answer to the above question. Moreover, our solution is constructive, i.e. given a lattice $L$, we are able to give a concrete set of generating aggregation functions on $L$ from which each aggregation function on $L$ can be obtained by applying an uniform method.

Our approach relies on the observation that from the point of view of universal algebra, aggregation functions on lattices form a clone (or, equivalently, a composition-closed set of functions containing projections). We show that for any finite lattice $L$ with $n$ elements this clone is finitely generated, and, moreover, we present explicitly at most $(2n+2)$-element generating set of aggregation functions, consisting of lattice binary operations, at most $n$ unary operations (certain characteristic functions) 
and at most $n$ binary aggregation functions (certain $0,1$-testing functions). Our construction shows that any aggregation function on $L$ arises as the usual term composition of generating functions. In other words, we show that in fact there is no need to use several quite different construction methods, but applying an uniform one, the composed aggregation, we obtain all the aggregation functions. We also stress the fact that our generating set of aggregations functions is very simple one, and its cardinality grows only linearly with respect to the cardinality of a lattice. Moreover, as our approach does not depend on the 
cardinality of $L$, any aggregation function on a lattice of infinite cardinality (e.g. such as a bounded real interval 
$[a,b]$) can be obtained as (possibly infinite) supremum of the generating set (of the same cardinality as a lattice) of aggregation functions. 

\section{Algebras, clones and near-unanimity functions}

First of all, we recall some necessary concepts from universal algebra, cf. \cite{Burris} or \cite{McKenzie}. By an {\it algebra} we mean a structure 
$(A;F)$, where $A$ is a nonempty set (called the {\it support} of the algebra) and $F$ is a (possibly empty) set of operations on $A$. If there is no danger of confusion, we usually do not distinguish between the algebra and its support.
To simplify expressions, if $f\in F$ is an $n$-ary operation and $\mathbf{x}=(x_1,\dots,x_n)\in A^n$, the 
evaluation of $f$ in $\mathbf{x}$ will be denoted by $f(\mathbf{x})$. 
 
A nonempty subset $B$ of $A$ is called a 
{\it subalgebra} of $(A;F)$ whenever $B$ is closed under all operations of $F$, i.e. if $f\in F$  is an $n$-ary 
operation, then $f(\mathbf{x})\in B$ for any $\mathbf{x}\in B^n$. Clearly, then the structure $(B;F)$ is again the algebra where 
the operations on $B$ are just those on $A$ but restricted to $B$. 

Further, given an algebra $(A;F)$ and $d\in \mathbb N$, by the {\it direct} $d$-th {\it power} of $(A;F)$ we mean an algebra 
$(A^d,F)$ with the support being the Cartesian $d$-th power of $A$, and the operations defined component-wise, i.e. 
for any $n$-ary operation symbol $f\in F$ and any $n\times d$ matrix $(a_{ij})$ of elements of $A$ we have
$$f((a_{11},\dots,a_{1d}),\dots,(a_{n1},\dots,a_{nd}))=(f(a_{11},\dots,a_{n1}),\dots,f(a_{1d},\dots,a_{nd})).$$
\noindent
We say that a $k$-ary function $g$ on $A$ {\it preserves the subalgebras of the direct $d$-th power} $(A^d,F)$ if 
for any subalgebra $B$ of $(A^d,F)$, whenever we have a $d\times k$ matrix $(b_{ij})$ of elements of $A$ all the columns of 
which belong to $B$, then so does the $d$-tuple when applying $g$ to its rows:

$$ 
\Spvek[c]{b_{11};b_{21};\vdots;b_{d1}},\Spvek[c]{b_{12};b_{22};\vdots;b_{d2}},\dots, \Spvek[c]{b_{1k};b_{2k};\vdots;b_{dk}}\in B \ \ \Longrightarrow\ \ \Spvek[c]{g(b_{11},b_{12},\dots,b_{1k});g(b_{21},b_{22},\dots,b_{2k});\vdots; g(b_{d1},b_{d2},\dots,b_{dk})}\in B.
$$      

The notion of a clone generalizes that of a monoid in a sense that it can be viewed as a set of selfmaps of 
a set $A$ that is closed under composition and containing the identical mapping. 
For an overview of clone theory we refer to \cite{Cs}, \cite{KP}, \cite{KPS} or \cite{Lau}.

Particularly, a {\it clone} on a set $A$ is a set of (finitary) operations on $A$ which contains all of the projection operations on $A$ and is closed under the composition, where projections and composition are formally defined as follows:\\
\noindent
Let $A$ be a set and $n\in\mathbb N$ be a positive integer. For any $i\leq n$, the {\it $i$-th $n$-ary projection} 
is for all $x_1,\dots,x_n\in A$ defined by 
$$p_i^n(x_1,\dots,x_n):=x_i.$$
Composition forms from one $k$-ary operation $f:A^k\mapsto A$ and $k$ $n$-ary operations $g_1,\dots,g_k:A^n\mapsto A$, 
an $n$-ary operation $f(g_1,\dots,g_k):A^n\mapsto A$ defined by
\begin{equation}\label{eq2}
f\big(g_1,\dots,g_k\big)(x_1,\dots,x_n):=f\big(g_1(x_1,\dots,x_n),\dots,g_k(x_1,\dots,x_n)\big),
\end{equation}
for all $x_1,\dots,x_n\in A$. 
For $k=n=1$, composition is a usual product of selfmaps.

Clones as sets of functions can be viewed by another equivalent way, namely, as the sets of (finitary) relations on $A$ that are 
preserved by all of the functions from the clone. More precisely, let $\rho\subseteq A^h$ be an $h$-ary relation on $A$,
and $f:A^n\mapsto A$ an $n$-ary operation on $A$. We say that $f$ {\it preserves} $\rho$ (or $\rho$ is {\it invariant} with respect to $f$), if for any $h\times n$ matrix of 
elements of $A$, if each of the $n$ columns $\mathbf{c}_1,\dots,\mathbf{c}_n$ of the matrix belongs to $\rho$, then the application 
of $f$ to the rows $\mathbf{r}_1,\dots,\mathbf{r}_h$ of the matrix also belongs to $\rho$, i.e. $(f(\mathbf{r}_1),\dots,f(\mathbf{r}_h))\in\rho$, which fact will be denoted by $f\triangleleft \rho$.  

Clearly, $\triangleleft$ defines a binary relation between 
the set ${\mathcal O}_A$ of all (finitary) functions on $A$ and ${\mathcal R}_A$, the set of all (finitary) relations on $A$.
It is well known that any binary relation induces a Galois connection between the corresponding sets, i.e., for any set 
$F\subseteq {\mathcal O}_A$ of functions on $A$ and $R\subseteq {\mathcal R}_A$ a set of relations on $A$, we consider the sets 
$$\mathsf{Inv}\, F:=\{\rho\in {\mathcal R}_A:\,\,f\triangleleft \rho \;\text{for any}\; f\in F\},$$ 
and
$$\mathsf{Pol}\, R:=\{f\in {\mathcal O}_A:\,\,f\triangleleft \rho \;\text{for any}\; \rho\in R\}.$$ 
In other words, the set $\mathsf{Inv}\, F$ consists of relations {\it invariant} with respect to all functions $f\in F$, and dually, 
$\mathsf{Pol}\, R$, called the set of {\it polymorphisms} of $R$, consists of operations to which are invariant all of the relations of $R$.
To simplify notation, we shall write $\mathsf{Pol}\,\rho$ instead of $\mathsf{Pol}\{\rho\}$ whenever $R=\{\rho\}$ is a singleton.
Consequently, combining the maps $\mathsf{Inv}$ and $\mathsf{Pol}$, we obtain a pair of closure operators defined on ${\mathcal O}_A$ and ${\mathcal R}_A$ by
$$F\mapsto \mathsf{Pol}\,\mathsf{Inv}\, F,$$
$$R\mapsto \mathsf{Inv}\,\mathsf{Pol}\, R.$$ 
It is well known and easy to see that the clones (of functions) are just the closed sets with respect to the closure operator $\mathsf{Pol}\,\mathsf{Inv}$, i.e. $\mathcal{C}=\mathsf{Pol}\,\mathsf{Inv}\, \mathcal{C}$, and thus can be viewed as sets of functions invariant with respect to appropriate sets of finitary relations. The clone $\mathsf{Pol}\,\mathsf{Inv}\,F$ is the least clone containing the set $F$ of functions, and thus called the clone {\it generated} by $F$. 
We call a clone {\it finitely generated} if is has a finite generating set of functions. Let us remark that closed sets with respect to the dual closure operator $\mathsf{Inv}\,\mathsf{Pol}$ are called {\it relational clones}, and both the closed sets form (with respect to set inclusion) dually isomorphic complete lattices. It is known that there are countably many clones on a two-element
set, and their lattice is completely understood and described since the work by E. Post \cite{Post}. However, there is a continuum clones 
on a set with at least of three elements and a full description of this lattice seems to be hopeless.

Yet another definition of a clone from the point of view of universal algebra: clones on $A$ are just term operations of algebras on $A$. This easily follows from the fact that by definition the terms of an algebra contain the projections and are composed in the same way as functions in clones.       

Evidently, the least clone on $A$ (called the {\it trivial clone} on $A$) contains just all the projections, the greatest one (called the {\it full clone} on $A$)
consists of all the functions on $A$, i.e. it coincides with ${\mathcal O}_A$.
Clones which are covered by the full clone, called {\it maximal} clones,  are of a particular importance. It has been proved that 
for a finite set $A$, any clone $\mathcal{C}\neq {\mathcal O}_A$ on $A$ is contained in a maximal one, and by famous Rosenberg
classification \cite{Ros}, all maximal clones on $A$ are of the form $\mathsf{Pol}\,\rho$ for six types of single relations $\rho$ on $A$. 
One of these six types are bounded orders on $A$, i.e. for any bounded order relation $\leq$ on $A$ (i.e. with the least and the greatest element), the clone $\mathsf{Pol}\leq$ is maximal. These clones are referred to as {\it monotone} clones on $A$. Clearly, for an $n$-ary function $f$ on $A$ we have $f\in \mathsf{Pol}\leq$ if and only if for any $\mathbf{x},\mathbf{y}\in A^n$, $\mathbf{x}\leq \mathbf{y}$ yields $f(\mathbf{x})\leq f(\mathbf{y})$.  

All the remaining five types of maximal clones are known to be finitely generated, cf. \cite{Lau}. For monotone clones, it is not always the 
case. It was quite surprising when Tardos in 1986 \cite{Ta} presented the 8-element partially ordered set (see Figure \ref{fig1}) the monotone clone of which is not finitely generated.

\begin{figure}
\begin{center}
\includegraphics{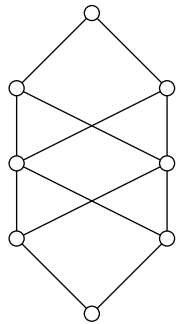}
\caption{A poset, the monotone clone of which is not finitely generated}
\label{fig1}
\end{center}
\end{figure}

On the other hand, clones on a finite set $A$ containing the so-called near-unanimity function are known to be finitely 
generated. Recall that for $n\geq 3$, an $n$-ary function $f$ on $A$ is called a {\it near-unanimity function} if 
$$f(x,\dots,x,y,x,\dots,x)=x$$  
for any $x,y\in A$, i.e., if any $n-1$ of the $n$ inputs coincide, the output of $f$ takes the same value. For $n=3$ 
we have the identities 
$$f(x,x,y)=f(x,y,x)=f(y,x,x)=x$$
for any $x,y\in A$, and these functions are called the {\it majority functions} on $A$.

Typically, if $(L;\vee,\wedge)$ is a lattice, then the functions defined by
\begin{equation}\label{eq3a} 
f(x,y,z):=(x\wedge y)\vee (y\wedge z)\vee (x\wedge z)\stepcounter{equation}\tag{{\theequation}a}
\end{equation} 
or
\setcounter{equation}{2}
\begin{equation}\label{eq3b}
g(x,y,z):=(x\vee y)\wedge (y\vee z)\wedge (x\vee z) \stepcounter{equation}\tag{{\theequation}b}
\end{equation}

are apparently the monotone majority functions on $L$.

That the clones on finite sets containing a near-unanimity function are finitely generated can be deduced 
from the following theorem by Baker and Pixley, see \cite{BP}:

\begin{theorem} {\rm(}Baker-Pixley Theorem{\rm )}\\
Let $\mathcal{C}$ be clone on a finite set $A$, let $f\in {\mathcal O}_A$ and assume that $\mathcal{C}$ contains a $(d+1)$-ary near-unanimity function. Then $f\in \mathcal{C}$ if and only if 
$f$ preserves all subalgebras of the direct $d$-th power of the algebra $(A;\mathcal{C})$.
\end{theorem}

\begin{corollary}\label{cor1}
Let $\mathcal{C}$ be a clone on a finite set $A$ containing a near-unanimity function. Then $\mathcal{C}$ is finitely generated.
\end{corollary}
\begin{proof}
By Baker-Pixley Theorem, a function $f\in {\mathcal O}_A$ belongs to $\mathcal{C}$ if and only if $f$ preserves all subalgebras of the direct $d$-th power of the algebra $(A;\mathcal{C})$. Since $A$ is finite, there are only finitely many subalgebras of $(A^d;\mathcal{C})$, the 
same holds for non-subalgebras $B$ of $(A^d;\mathcal{C})$. For any such $B$ consider a function $f_B\in \mathcal{C}$ such that 
$f_B$ does not preserve $B$. By previous arguments, such a function must exists.
 
We show that the collection of all the functions $f_B$ together with a near-unanimity function yields a finite generating set of the clone $\mathcal{C}$. Indeed, denote by $\mathcal{D}$ the clone generated by a near-unanimity function and by the functions $f_B$. Clearly, we have the inclusion $\mathcal{D}\subseteq\mathcal{C}$.
We claim that the algebras $(A^d;\mathcal C)$ and $(A^d;\mathcal D)$ have exactly the same subalgebras.
Obviously, as $\mathcal D\subseteq \mathcal C$, $(A^d;\mathcal C)$ has possibly less subalgebras than $(A^d;\mathcal D)$. 
But by the construction of $\mathcal D$, for any non-subalgebra $B$ of $(A^d;\mathcal C)$ there is a function $f_B\in \mathcal{D}$ such that $f_B$ does not preserve $B$, i.e. $B$ is a non-subalgebra of $\mathcal D$ as well. 

Finally, as $(A^d;\mathcal C)$ and $(A^d;\mathcal D)$ have the the same subalgebras, applying Baker-Pixley Theorem we obtain the desired equality $\mathcal C=\mathcal D$.
\end{proof}

Consequently, for any finite set $A$, monotone clones of the form $\mathcal{C}=\mathsf{Pol}\leq$, where $\leq$ is a bounded lattice order on $A$ with 
$0$ and $1$ as the least and the greatest element, 
are finitely generated. The same concerns the clones $\mathsf{Pol}\{\leq,0,1\}$ consisting of all monotone functions of bounded lattice orders $\leq$ preserving the bounds $0$ and $1$, i.e. fulfilling the boundary conditions
\begin{equation}\label{eq4} 
f(0,\dots,0)=0 \quad \text{ and } \quad f(1,\dots,1)=1.
\end{equation} 
This follows from the fact that the majority lattice functions given by \eqref{eq3a} or \eqref{eq3b} satisfy the boundary conditions \eqref{eq4}.  

We will see that these clones correspond just to aggregation functions on bounded posets, thus we refer to them as {\it aggregation clones} on bounded posets.

It can be seen from Corollary \ref{cor1} that although the aggregation clones on finite bounded lattices are finitely generated, the proof of this fact does not give any simple algorithm how to produce their generating sets, as well as it does not yield to any a priori estimates of arities of generating functions. In what follows we fill this gap by showing that for an $n$-element lattice, always at most $2n+2$ generating functions are enough with arities bounded by 2.

\section{Aggregation functions on lattices}

Recall that an agreggation function on $[0,1]$ is a function $A: [0,1]^n\to [0,1]$ which is monotone increasing, i.e. 
$A(\mathbf{x})\leq A(\mathbf{y})$ whenever $\mathbf{x}\leq \mathbf{y}$, (i.e., $x_i\leq y_i$ for all $i\in\{1,\dots,n\}$), and 
$A(0,\dots,0)=0, A(1,\dots,1)=1$. Obviously, the framework of aggregation functions can be modified by considering functions on any closed real interval, and clearly to any partially ordered structure with bounds (see e.g. \cite{KM}):

\begin{definition}
Let $(P,\leq,0,1)$ be a bounded partially ordered set (poset), let $n\in \mathbb N$. A mapping $A: P^n\to P$ is called 
an ($n$-ary) aggregation function on $P$ if it is monotone increasing, i.e. for any $\mathbf{x}, \mathbf{y}\in P^n$: 
$$A(\mathbf{x})\leq A(\mathbf{y})\,\, \text{ whenever }\,\, \mathbf{x}\leq \mathbf{y},$$
and it satisfies boundary conditions 
$$A(0,\dots,0)=0 \quad\text{ and }\quad A(1,\dots,1)=1.$$
\end{definition}  

For a more detailed discussion on aggregation functions on posets or lattices we recommend the paper by Demirci \cite{Dem}. 
Special types of aggregation functions on posets, especially triangular norms or conorms, are studied e.g. in \cite{DeBM,M2,M1}. 
It is easy to see that considering $P=[0,1]$ to be the standard interval of reals with the usual ordering, we obtain the   
classical definition of an aggregation function. A particular example of a binary aggregation function on a $3$-element chain is schematically shown on Figure \ref{fig3}.

\begin{figure}
\begin{center}
\includegraphics[scale=1.25]{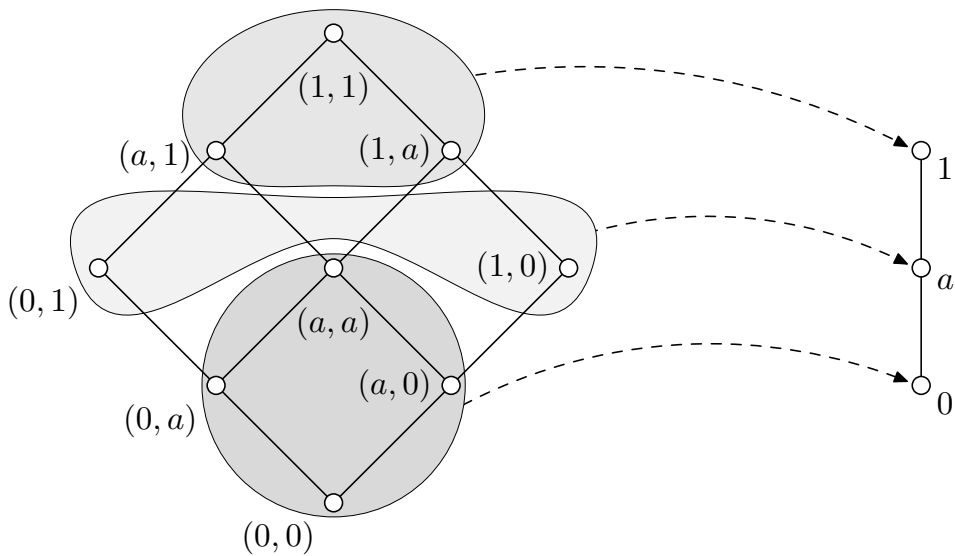}
\caption{A binary aggregation function on a three-element chain}
\label{fig3}
\end{center}
\end{figure}

From the point of view of clone theory, aggregation functions on bounded posets are just the members of the corresponding aggregation clone. Already the fact that they form a clone has many important consequences: for example, we obtain "for free" some of their construction methods known as composed aggregation (see e.g. the monograph \cite{GMRP}). These are just special instances of that obtained from a composition of functions in a clone \eqref{eq2}.
The following illustrative example modifies one of the methods called the double aggregation: 

\begin{example}
Let $P$ be a bounded poset, let $A: P^2\mapsto P$, $B: P^n\mapsto P$ and $C: P^m\mapsto P$ be aggregation 
functions on $P$. Then the function $D_{A;B,C}: P^{n+m}\mapsto P$ defined by
$$D_{A;B,C}(x_1,\dots,x_{n+m}):=A(B(x_1,\dots,x_n),C(x_{n+1},\dots,x_{n+m}))$$
is an $(n+m)$-ary aggregation function on $P$.\\
Indeed, we may consider the functions $B$ and $C$ as $(n+m)$-ary, $B$ depending only on the first $n$ variables, 
$C$ depending on the last $m$ variables. Then applying \eqref{eq2} for $f:=A, g_1:=B$ and $g_2:=C$, we obtain just the 
function $D_{A;B,C}$.
\end{example}

In order to generate the aggregation clone $\mathcal{C}_L$ of a finite lattice $L$, we use the following unary and binary functions: 

For any $a\in L$ we define $\chi_a\colon L\to L$ by
\begin{equation}\label{e1}
\chi_a(x)=
\begin{cases}&1, \text{ if }x\geq a, x\neq 0; \\
             &0, \text{ otherwise.}\\   
\end{cases}
\end{equation} 

Obviously, $\chi_a$ is an aggregation function for all $a\in L$. Moreover, it represents a characteristic function of the principal filter $F(a)=\{x\in L:x\geq a\}$ generated by $a$, provided $a\neq 0$. 

Further, for $b\in L$ we define a binary function $\oplus_b\colon L\times L\to L$ by
\begin{equation}\label{e2}
x\oplus_b y=\begin{cases}&1, \text{ if }x=1,y=1; \\
                         &0, \text{ if }x=0,y=0;\\   
                         &b, \text{ otherwise.}    
              \end{cases}
\end{equation}

Let us note that we prefer the infix notation $x\oplus_b y$ for this family of operations, since any of them is associative and commutative, as it can be easily verified. Also, for any $b\in L$ the function $\oplus_b$ is monotone and satisfies the boundary conditions, i.e., it is a binary aggregation function on $L$.   
As it is common for the associative binary operations, for a positive integer $n\geq 2$ the symbol  
$$ \bigoplus_{i=1}^n{\hspace{-0.6mm}_b}\ x_i=x_1\oplus_b x_2\oplus_b \dots \oplus x_{n-1}\oplus_b x_n$$ will denote the repeated composition (in an arbitrary order) of the operator $\oplus_b$.

Let $b\in L$ be an element such that $0<b<1$. Then 
$$ \bigoplus_{i=1}^n{\hspace{-0.6mm}_b}\ x_i \neq b \quad\text{ iff }\quad (x_1,\dots,x_n)=(0,\dots,0)\ \text{ or }\ (x_1,\dots,x_n)=(1,\dots,1).$$

For $\mathbf{a}=(a_1,\dots,a_n)\in L^n$ we denote by $J_{\mathbf{a}}=\{1\leq i\leq n: a_i\neq 0\}$ the set of non-zero indices 
of $\mathbf{a}$.

Let $n\geq 2$ be an integer and $f\colon L^n\to L$ be an aggregation function. For any $\mathbf{a}\in L^n$ with $(0,\dots,0)<\mathbf{a}<(1,\dots,1)$ (and, consequently, $J_{\mathbf{a}}\not =\emptyset$) we define the function
$h_{\mathbf{a}}\colon L^n\to L$ by putting

\begin{equation}\label{e3}
h_{\mathbf{a}}(x_1,\dots,x_n)=\bigwedge_{i\in J_{\mathbf{a}}}\chi_{a_i}(x_i)\; \wedge\; \bigoplus_{i=1}^{n}{\hspace{-1mm}_{f(\mathbf{a})}}\ x_i.
\end{equation}

\noindent
The functions $h_{\mathbf{a}}$ can be described as follows: 

\begin{lemma}\label{lem1}
The function $h_{\mathbf{a}}$ is an aggregation function being composed of some aggregation functions defined by $\eqref{e1}$ and $\eqref{e2}$, respectively. Moreover, for all $\mathbf{x}=(x_1,\dots,x_n)\in L^n,$
\begin{equation}\label{e4}
h_{\mathbf{a}}(\mathbf{x})=\begin{cases}&1, \text{ if }\ \mathbf{x}=(1,\dots,1);\\
                                        &f(\mathbf{a}),\text{ if }\ \mathbf{x}\geq \mathbf{a}, (1,\dots,1)\neq\mathbf{x}\neq (0,\dots,0);\\   
                                        &0, \text{ if }\ \mathbf{x}\ngeq\mathbf{a} \text{ or } \mathbf{x}=(0,\dots,0).    
              \end{cases}
\end{equation}
\end{lemma}

\begin{proof}
Obviously, $h_{\mathbf{a}}$ is a composition of some aggregation functions defined by (\ref{e1}) and (\ref{e2}). This yields that $h_{\mathbf{a}}(0,\dots,0)=0$ as well as $h_{\mathbf{a}}(1,\dots,1)=1$.

Further, assume that $(1,\dots,1)\neq \mathbf{x}\geq\mathbf{a}$. Recall that $\mathbf{x}\geq\mathbf{a}$ if and only if $x_i\geq a_i$ for all $i\in J_{\mathbf{a}}$. In this case, $\chi_{a_i}(x_i)=1$ for each index $i\in J_{\mathbf{a}}\neq\emptyset$, since no $a_i$ is equal to zero. Also $(1,\dots,1)\neq \mathbf{x}\neq (0,\dots,0)$ implies $\bigoplus\limits_{i=1}^{n}{\hspace{-1.2mm}_{f(\mathbf{a})}} x_i=f(\mathbf{a})$. Hence we obtain 
$$ h_{\mathbf{a}}(\mathbf{x})=\bigwedge_{i\in J_{\mathbf{a}}}\chi_{a_i}(x_i)\; \wedge\; \bigoplus_{i=1}^{n}{\hspace{-1mm}_{f(\mathbf{a})}}\ x_i=1\wedge f(\mathbf{a})= f(\mathbf{a}).$$

If $\mathbf{x}\ngeq\mathbf{a}$, then $x_i\ngeq a_i$ for some index $i\in J_{\mathbf{a}}$. Consequently $\chi_{a_i}(x_i)=0$, which yields 
$$ h_{\mathbf{a}}(\mathbf{x})=\bigwedge_{i\in J_{\mathbf{a}}}\chi_{a_i}(x_i)\; \wedge\; \bigoplus_{i=1}^{n}{\hspace{-1mm}_{f(\mathbf{a})}}\ x_i=0\wedge \bigoplus_{i=1}^{n}{\hspace{-1mm}_{f(\mathbf{a})}}\ x_i= 0.$$

\end{proof}

Denote by $L^n_{*}$ the set of all elements between the bottom and the top element of $L^n$, i.e., $L^n_{*}=L^n\setminus\{(0,\dots,0),(1,\dots,1)\}$. 

\begin{lemma}\label{lem2}
Let $f\colon L^n \to L$ be an aggregation function and for all $\mathbf{a}\in L^n_{*}$, $h_{\mathbf{a}}$ be the function defined by $(\ref{e3})$. Then 
\begin{equation}\label{e41}
f(\mathbf{x})=\bigvee_{\mathbf{a}\in L^n_{*}}h_{\mathbf{a}}(\mathbf{x})
\end{equation}
for all $\mathbf{x}\in L^n$. 
\end{lemma}  

\begin{proof}
Let $\mathbf{x}\in L^n$ be an arbitrary element. By previous lemma, if $\mathbf{x}=(0,\dots,0)$ or $\mathbf{x}=(1,\dots,1)$, then $\bigvee_{\mathbf{a}\in L^n_{*}}h_{\mathbf{a}}(\mathbf{x})$ gives the corresponding boundary values. Thus we may assume that $\mathbf{x}\in L^n_{*}$. Using (\ref{e4}) of Lemma \ref{lem1} we obtain
$$ \bigvee_{\mathbf{a}\in L^n_{*}} h_{\mathbf{a}}(\mathbf{x})=\bigvee_{\substack{\mathbf{a}\in L^n_{*}\\ \mathbf{a}\leq \mathbf{x}}} h_{\mathbf{a}}(\mathbf{x})\ \vee \ \bigvee_{\substack{\mathbf{a}\in L^n_{*}\\ \mathbf{a}\nleq \mathbf{x}}} h_{\mathbf{a}}(\mathbf{x})= \bigvee_{\substack{\mathbf{a}\in L^n_{*}\\ \mathbf{a}\leq \mathbf{x}}} f(\mathbf{a}) \ \vee \ \bigvee_{\substack{\mathbf{a}\in L^n_{*}\\ \mathbf{a}\nleq \mathbf{x}}} 0= \bigvee_{\substack{\mathbf{a}\in L^n_{*}\\ \mathbf{a}\leq \mathbf{x}}} f(\mathbf{a}).$$
Since the function $f$ is monotone and $\mathbf{x}$ represents the greatest element of the set $\{\mathbf{a}\in L^n_{*}: \mathbf{a}\leq x\}$, it follows that $f(\mathbf{a})\leq f(\mathbf{x})$ for all $\mathbf{a}\in L^n_{*}$, $\mathbf{a}\leq \mathbf{x}$. Hence 
$$ \bigvee_{\mathbf{a}\in L^n_{*}} h_{\mathbf{a}}(\mathbf{x})= \bigvee_{\substack{\mathbf{a}\in L^n_{*}\\ \mathbf{a}\leq \mathbf{x}}} f(\mathbf{a})=f(\mathbf{x}),$$ which completes the proof.

\end{proof}

As a consequence of this lemma we obtain the following main result:

\begin{theorem}\label{thm1}

Let $L$ be a finite lattice. Then the aggregation clone $\mathcal{C}_L$ on $L$ is generated by the lattice operations $\vee$, $\wedge$, by functions $\chi_a$, $a\in L$, defined by $(\ref{e1})$ and by functions $\oplus_b$, $b\in L$, defined by $(\ref{e2})$. 

\end{theorem} 

Let us note that the duality principle for lattices allows to describe another generating set of aggregation functions. For an element $a\in L$, define a function $\mu_a\colon L\to L$ by
\begin{equation}\label{eqn1}
\mu_a(x)=\begin{cases}&0, \text{ if }x\leq a, x\neq 1; \\
             &1, \text{ otherwise.}\\   
\end{cases}
\end{equation} 

For $\mathbf{a}=(a_1,\dots,a_n)\in L^n$ we denote by $\hat{J}_{\mathbf{a}}=\{1\leq i\leq n: a_i\neq 1\}$.

Similarly, given an integer $n\geq 2$ and $f\colon L^n\to L$ an aggregation function, for any $\mathbf{a}\in L^n_*$ we define the function
$g_{\mathbf{a}}\colon L^n\to L$ by 

\begin{equation}\label{eqn2}
g_{\mathbf{a}}(x_1,\dots,x_n)=\bigvee_{i\in \hat{J}_{\mathbf{a}}}\mu_{a_i}(x_i)\; \vee\; \bigoplus_{i=1}^{n}{\hspace{-1mm}_{f(\mathbf{a})}}\ x_i.
\end{equation}

An analogous assertion to Lemma \ref{lem1} holds for the functions $g_{\mathbf{a}}$, $\mathbf{a}\in \hat{J}_{\mathbf{a}}$. In this case
\begin{equation}\label{eqn3}
g_{\mathbf{a}}(\mathbf{x})=\begin{cases}&0, \text{ if }\ \mathbf{x}=(0,\dots,0);\\
                                        &f(\mathbf{a}),\text{ if }\ \mathbf{x}\leq \mathbf{a}, (1,\dots,1)\neq\mathbf{x}\neq (0,\dots,0);\\   
                                        &1, \text{ if }\ \mathbf{x}\nleq\mathbf{a} \text{ or } \mathbf{x}=(1,\dots,1),    
              \end{cases}
\end{equation}
and $f$ can be expressed as
\begin{equation}\label{eqn41}
f(\mathbf{x})=\bigwedge_{\mathbf{a}\in L^n_{*}}g_{\mathbf{a}}(\mathbf{x})
\end{equation}
for all $\mathbf{x}\in L^n$. Hence, the aggregation clone $\mathcal{C}_L$ can be generated by the set of functions consisting of the lattice operations and by the functions defined by (\ref{eqn1}) and (\ref{e2}). 

It is worth noticing that the formulae (\ref{e41}) and (\ref{eqn41}) remain valid also if $L$ is infinite. In contrast to a finite case, there are infinite suprema in (\ref{e41}), and infinite infima in (\ref{eqn41}) respectively, which can be understood as a kind of limit process. This fact is especially important when considering a classical case, i.e. $L$ being the unit interval of reals. Obviously, this case deserves a much deeper study on its own and it will be the objective of the future research.  

We illustrate the proposed method of generating the aggregation clone on the following example.

\begin{example}\label{example3}
Denote by $f$ the aggregation function depicted in Figure \ref{fig3}. Recall that $f\colon L^2\to L$ where $L$ is the three-element chain with $0<a<1$. 
In order to generate the function $f$ we have to define the corresponding binary functions $h_{\mathbf{a}}$ for all $\mathbf{a}\in L^2_{*}$. 
Denoting the variables by $x$ and $y$, respectively, expression \eqref{e3} gives the following seven functions:
\begin{equation*}
\begin{aligned}
 & h_{(0,a)}(x,y)=\chi_{a}(y)\wedge (x\oplus_0 y),\ h_{(a,0)}(x,y)=\chi_{a}(x)\wedge (x\oplus_0 y), \\  
 & h_{(a,a)}(x,y)=(\chi_{a}(x)\wedge \chi_a(y))\wedge (x\oplus_0 y),\ h_{(0,1)}(x,y)=\chi_{1}(y)\wedge (x\oplus_a y),\\ 
 & h_{(1,0)}(x,y)=\chi_{1}(x)\wedge (x\oplus_a y),\ h_{(a,1)}(x,y)=(\chi_{a}(x)\wedge \chi_1(y))\wedge (x\oplus_1 y),\\
 & h_{(1,a)}(x,y)=(\chi_{1}(x)\wedge\chi_{a}(y))\wedge (x\oplus_1 y).
\end{aligned}
\end{equation*}
Consequently, applying Lemma \ref{lem2}, we obtain the desired expression of the function $f$:  
\begin{equation*}
\begin{aligned}
f(x,y)&=\big( \chi_{a}(y)\wedge (x\oplus_0 y) \big)\vee \big( \chi_{a}(x)\wedge (x\oplus_0 y)\big) \vee \big( \chi_{1}(y)\wedge (x\oplus_a y)\big) \\
 &\vee \big((\chi_{a}(x)\wedge \chi_a(y))\wedge (x\oplus_0 y) \big) \vee \big( \chi_{1}(x)\wedge (x\oplus_a y)\big) \\
&\vee \big( (\chi_{a}(x)\wedge \chi_1(y))\wedge (x\oplus_1 y)\big)\vee \big( (\chi_{1}(x)\wedge\chi_{a}(y))\wedge (x\oplus_1 y)\big).
\end{aligned}
\end{equation*}
\end{example}

In connection with the above described generating sets, one can naturally ask whether it is possible to generate the aggregation clone without certain types of binary functions. We show that unary aggregation functions together with the lattice operations do not generate $\mathcal{C}_L$. In particular, this implies that at least some $n$-ary function, $n\geq 2$, different from the lattice operations, must be used in order to generate the full clone. 

\begin{theorem}
Let $L$ be a lattice with at least three elements. Then the set consisting of the lattice operations and all unary aggregation functions does not generate the aggregation clone $\mathcal{C}_L$. 
\end{theorem}

\begin{proof}

Let $L$ be a lattice with at least three elements. Denote by $\mathcal{D}$ the clone generated by the set of all unary aggregation functions of $L$,  together with the lattice operations $\vee$ and $\wedge$. Since $\mathcal{D}$ contains the lattice operations, it follows that the majority functions given by \eqref{eq3a} and \eqref{eq3b} belong to $\mathcal{D}$. These functions represent $3$-ary near-unanimity functions on $L$, hence by Baker-Pixley theorem, a function $f\in {\mathcal O}_L$ belongs to $\mathcal{D}$ if and only if $f$ preserves all the subalgebras of the direct square $(L^2;\mathcal{D})$.

Consider the subset $B=\{(1,0),(0,0)\}$ of $L^2$. Any unary aggregation function fulfils the boundary conditions, thus it is easily seen that $B$ is closed with respect to all unary aggregation functions of $L$. Evidently, $B$ is closed with respect to lattice operations $\vee$ and $\wedge$, hence $B$ is a subalgebra of $(L^2;\mathcal{D})$. 

Let $a\in L$ be an arbitrary element with $0<a<1$. The function $\oplus_a$ does not preserve $B$: indeed, we have $(1,0),(0,0)\in B$ and applying the function $\oplus_a$ we obtain $(1\oplus_a 0,0\oplus_a 0)=(a,0)\not\in B$. Consequently, as $\oplus_a$ is the aggregation function does not preserving all the subalgebras of $(L^2,\mathcal{D})$, it follows that $\mathcal{C}_L\neq\mathcal{D}$. 
\end{proof}

In the sequel, we try to present an application of our results in connection with some kind of optimality functions defined on networks, represented by median graphs. The most common problems studied in network location theory are the $p$-median problems. In this type of problems there are $n$ customers and the objective is to locate $p$ service facilities to minimize the sum of $n$ service distances, provided a customer is served only by the closest facility. 

Recall, that a graph $H=(V(H),E(H))$ is called a median graph, if it is connected and for every triple of its vertices there is a unique vertex $w$, called median, such that $w\in V(H)$ lies simultaneously on a shortest path between any two of them. Given a triple $(x,y,z)\in V(H)^3$ of vertices, for the corresponding unique median vertex $w$ the expression 
$$ d(w,x)+d(w,y)+d(w,z),$$ 
involving the graph metric $d$, attains a minimal value. Median graphs arise naturally in the study of networks, ordered sets and discrete distributive lattices, with many applications in various fields, e.g., in the network location theory or in human genetics, where they are used for an analysis of the mitochondrial DNA.
Figure \ref{fig2} shows an example of a median graph. The vertex $w$ denotes the median of three vertices $x,y$ and $z$. 

\begin{figure}
\begin{center}
\includegraphics[scale=1.25]{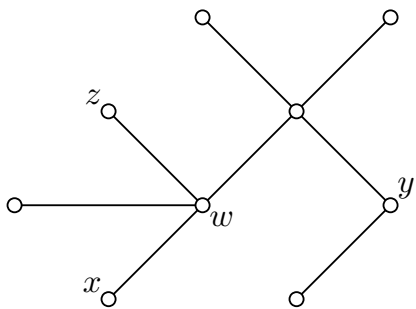}
\caption{An example of a median graph.}
\label{fig2}
\end{center}
\end{figure}

Further, we recall the notion of a retract of a graph. By a retraction $\varphi$ of a graph $G=(V(G),E(G))$, we mean a homomorphism of $G$ into itself with the property $\varphi^2(u) =\varphi(u)$ for all vertices $u\in V(G)$. The homomorphic image of $G$ under $\varphi$ is called a retract. Note, that if a vertex $v=\varphi(u)$ belongs to the retract, then $\varphi(v)=\varphi(\varphi(u))=\varphi(u)=v$, i.e., $\varphi$ restricted to the retract is the identity mapping.

A special class of median graphs represent hypercubes. The hypercube of dimension $r\geq 1$ is a graph, denoted by $Q_r$, isomorphic to that whose vertex set consists of all $0$-$1$ vectors $(v_1, v_2,\dots, v_r)$, where two vertices are adjacent if and only if they differ in precisely one coordinate. The hypercube $Q_3$ is depicted in Figure \ref{fig4}.

\begin{figure}
\begin{center}
\includegraphics[scale=1.25]{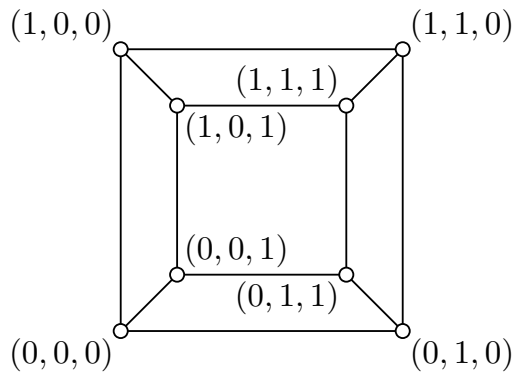}
\caption{The hypercube $Q_3$.}
\label{fig4}
\end{center}
\end{figure}

It is a well-known fact that median graphs can be characterized as retracts of hypercubes, cf. \cite{Ban84}. From the definition of a hypercube $Q_r$, it is easily seen that $Q_r$ is isomorphic to the covering graph of a Boolean algebra with $r$ atoms. This enables to extend the median function from a median graph to a Boolean algebra. Moreover, if $H$ is a retract of $Q_r$ via a retraction $\varphi\colon Q_r\to H$, $(x,y,z)$ is a triple of vertices from $H$ and $w$ is the median in $Q_r$, then $w$ belongs to the retract $H$. Indeed, since $\varphi$ is an edge preserving mapping, for any vertices $u,v\in Q_r$ it can be easily seen that $d(u,v)\geq d(\varphi(u),\varphi(v))$. Hence $$d(w,x)+d(w,y)+d(w,z)\geq d(\varphi(w),\varphi(x))+d(\varphi(w),\varphi(y))+d(\varphi(w),\varphi(z))$$ $$=d(\varphi(w),x)+d(\varphi(w),y)+d(\varphi(w),z).$$ Since $w$ is the unique vertex minimizing the distance from the three vertices $x,y,z$, it follows that $\varphi(w)=w$, i.e., $w\in H$ as well. 

In the hypercube $Q_r$, the median of the vertices $x=(x_1,\dots,x_r)$, $y=(y_1,\dots,y_r)$ and $z=(z_1,\dots,z_r)$ is computed by the ``majority" rule 
$$ f(x,y,z)=w=(w_1,\dots,w_r),\ w_i=\left[ \frac{x_i+y_i+z_i}{3}\right],$$ $[\cdot]$ denoting the ceiling function,
i.e., $w_i=1$ if at least two values from $x_i,y_i$ and $z_i$ are equal to $1$ and $w_i=0$ otherwise.
Considering $Q_r$ as a Boolean algebra with the componentwise order, it is easily seen that $f$ fulfills the boundary conditions and it is non-decreasing in each coordinate, i.e., it is an aggregation function.

This aggregation naturally defines a system of binary aggregation functions $f_v\colon Q_r^2\to Q_r$, $v\in Q_r$, where $f_v(x,y)=f(v,x,y)$. Involving a system of binary functions instead of one ternary can be more convenient in some applications following certain computational aims. Obviously, using the functions $f_v$, for $v\in Q_r$ can be very useful in connection with some decision problems, e.g., in the so-called multi-facility location problems, where several types of median functions are widely involved. However, for other purposes the usage of the other types of aggregation functions can be more appropriate. Assume that also some procedures for deciding whether $x\in Q_r$ is above a given $v\in Q_r$ are available. This represents the system of the unary functions $\chi_v$, $v\in Q_r$.

Considering the functions $\chi_v$, $f_v$ for $v\in Q_r$ and the lattice operations on $Q_r$ as simple computation models, combining the outputs of these particular functions as inputs for other such functions, one can obtain a complex computational model. Naturally, it is important to know which aggregation functions on $Q_r$ can be obtained by such defined complex computational models. Mathematically, this question is equivalent to the problem of characterizing the clone generated by these functions. 

\begin{example}
As an application of Theorem \ref{thm1}, we show that the set of functions $G=\{\chi_v:v\in Q_r\}\cup \{f_v:v\in Q_r\}$ together with the lattice operations generates the full aggregation clone on $Q_r$. Obviously, it suffices to show that each $\oplus_v$, $v\in Q_r$ belongs to the clone generated by the set $G\cup\{\wedge,\vee\}$. For this, given $v\in Q_r$, consider $$g=f_v\big(\chi_{\mathbf{0}}(x\vee y),\chi_{\mathbf{1}}(x\wedge y)\big).$$ 
Analyzing all possibilities, one can easily see that $g(x,y)=\mathbf{0}$ if $x=y=\mathbf{0}$, $g(x,y)=\mathbf{1}$ if $x=y=\mathbf{1}$ and $g(x,y)=v$ otherwise, i.e., $g(x,y)=x\oplus_v y$ for all $x,y\in Q_r$.   

Consequently, the lattice operations, characteristic functions $\chi_v$ and the functions $f_v$, $v\in Q_r$, form a generating basis, i.e. any aggregation function on $Q_r$ can be obtained by (possibly iterated) composition of them.
\end{example}

\section{Conclusion}
In this paper we have shown that aggregation functions on finite lattices form a clone of functions which is 
finitely generated. Moreover, we presented an explicit description of generating sets of functions.
We believe that the proposed approach will be convenient also for an analysis of special classes of aggregation 
functions on lattices or even certain posets. 

In the future work we would like to extend this idea to the study of conjunctive, disjunctive or internal 
aggregation functions.

\section*{Acknowledgments}

The second author was supported by the ESF Fund CZ.1.07/2.3.00/30.0041 and by the Slovak VEGA Grant 2/0028/13, the first author by the international project Austrian Science Fund (FWF)-Grant Agency of the Czech Republic (GA\v{C}R) number I 1923-N25.



\begin{thebibliography}{00}

\bibitem{BP}
Baker K. A., Pixley A. F., Polynomial interpolation and the chinese remainder theorem for algebraic systems, Math. Zeitschrift 143, pp. 165-174, 1975.

\bibitem{Ban84}
Bandelt H.-J. Retracts of hypercubes, Journal of Graph Theory 8, pp. 501-510, 1984.

\bibitem{Burris}
Burris S., Sankappanavar H. P., A Course in Universal Algebra, Springer-Verlag, 1981.

\bibitem{Cs}
Cs\'ak\'any B., Minimal clones - a minicourse, Algebra Universalis 54, pp. 73-89, 2005.

\bibitem{DeBM}
De Baets B., Mesiar R., Triangular norms on product lattices, Fuzzy Sets and Systems 104, pp. 61-76, 1999.

\bibitem{Dem}
Demirci M., Aggregation operators on partially ordered sets and their categorical foundations, Kybernetika 42, 261-277, 2006.

\bibitem{GMRP}
Grabisch M., Marichal J.-L., Mesiar R., Pap E., Aggregation Functions, Cambridge University Press, Cambridge, 2009.

\bibitem{G1}
Gr\"atzer G., Lattice Theory: Foundation, Birkh\"auser, Basel, 2011.

\bibitem{G2}
Gr\"atzer G., Wehrung F. (Eds.), Lattice Theory: Special Topics and Applications, Vol 1, Birkh\"auser, Basel, 2014. 

\bibitem{KP}
Kaarli K., Pixley A. F., Polynomial completeness in algebraic systems, Chapman \& Hall / CRC, Boca Raton, Florida, 2001.

\bibitem{M2}
Karacal F., Mesiar R., Uninorms on bounded lattices, Fuzzy Sets and Systems 261, pp. 33-43, 2015.

\bibitem{KPS}
Kerkhoff S., P\"oschel R., Schneider F.M., A Short Introduction to Clones, Electronic Notes in Theoretical Computer Science 303, pp. 107–120, 2014. 

\bibitem{KM}
Komorn\'ikov\'a M., Mesiar R., Aggregation functions on bounded partially ordered sets and their classification, Fuzzy Sets and Systems 175, pp. 48-56, 2011.

\bibitem{Lau}
Lau D., Function algebras on finite sets, Springer-Verlag, Berlin, 2006.

\bibitem{McKenzie}
McKenzie R., McNulty  G., Taylor W., Algebras, Lattices and Varieties, Vol. I, Wadsworth \& Brooks/Cole, Monterey, California, 1987.

\bibitem{Post}
Post E. L., The Two-Valued Iterative Systems of Mathematical Logic, Annals of Mathematics Studies, no. 5, Princeton University Press, Princeton, N. J., 1941. 

\bibitem{Ros}
Rosenberg I. G., \"Uber die funktionale Vollst\"andigkeit in den mehrwertigen Logiken. Struktur der
Funktionen von mehreren Ver\"anderlichen auf endlichen Mengen, Rozpravy \v{C}eskoslovensk\'e Akad. V\v{e}d. \v{R}ada Mat. P\v{r}\'irod. V\v{e}d 80, pp. 3-93, 1970.

\bibitem{M1}
Saminger-Platz S., Klement E. P., Mesiar R.,On extensions of triangular norms on bounded lattices, Indagationes Mathematicae 19(1), pp. 135-150, 2008.

\bibitem{Ta}
Tardos G., A maximal clone of monotone operations which is not finitely generated, Order 3, pp. 211-218, 1986.

\end{thebibliography}
\end{document}